\documentclass{amsart}
\usepackage{amssymb}
\usepackage{amsmath}

\newtheorem{theorem}{Theorem}[section]
\newtheorem*{theorem A}{Theorem A}
\newtheorem*{theorem B}{N\"olker's Theorem}
\newtheorem{lemma}{Lemma}[section]
\newtheorem{proposition}{Proposition}[section]
\newtheorem{corollary}{Corollary}[section]

\theoremstyle{remark}

\theoremstyle{remark}

\theoremstyle{definition}

\numberwithin{equation}{section}
\def \({\left ( }
\def \){\right )}
\def \<{\left < }
\def \>{\right >}

\input{tcilatex}

\begin{document}
\title{\textbf{Semiparalel Wintgen Ideal Surfaces in }$\mathbb{E}^{n}$}
\author{Bet\"{u}l Bulca \& Kadri Arslan}
\address{Uluda\u{g} University, Art and Science Faculty, Department of
Mathematics, Bursa-TURKEY}
\email{bbulca@uludag.edu.tr; arslan@uludag.edu.tr}
\thanks{This paper is supported by Uluda\u{g} University Research found with
Project No: KUAP(F)-2012/59}
\date{May, 2013 }
\dedicatory{}

\begin{abstract}
Wintgen ideal surfaces in $\mathbb{E}^{4}$ form an important family of
surfaces, namely surfaces with circular ellipse of curvature. Obviously,
Wintgen ideal surfaces satisfy the pointwise equality $K+\left \vert
K_{N}\right \vert =\left \Vert H\right \Vert ^{2}.$ In the present study we
consider the Wintgen ideal surfaces in $n$-dimensional Euclidean space $%
\mathbb{E}^{n}.$ We have shown that Wintgen ideal surfaces in $\mathbb{E}%
^{n} $ satisfying the semiparallelity condition $\overline{R}(X,Y)\cdot h=0$
are totally umbilical. Further, we obtain some results in $\mathbb{E}^{4}.$
\end{abstract}

\maketitle

\section{Introduction}

In 1979, Wintgen \cite{Wi} proved a basic relation between the intrinsic
Gauss curvature $K$, the extrinsic normal curvature $K_{N}$ and a squared
mean curvature $\left \Vert H\right \Vert ^{2}$ of any surface $M$ in
Euclidean 4-space $\mathbb{E}^{4},$ namely 
\begin{equation*}
K+K_{N}\leq \left \Vert H\right \Vert ^{2}
\end{equation*}%
with the equality holding if and only if the curvature ellipse is a circle (%
\cite{DDVV}, \cite{GR}). Following Verstraelen et. al. \cite{DPV}, \cite{PV}%
, a surface $M$ in $\mathbb{E}^{4}$ is called Wintgen ideal if it satisfies
the equality case of Wintgen inequality identically. Obviously Wintgen ideal
surfaces in $\mathbb{E}^{4}$ are exactly superminimal \cite{Fr} or
superconformal \cite{DT} surfaces. In \cite{Ch2} B.Y. Chen completely
classified Wintgen ideal surfaces in $\mathbb{E}^{4}$ with equal Gauss and
normal curvatures.

Wintgen's inequality is extended to surfaces in real space form by I.V.
Guadalupe and L. Rodriguez \cite{GR}. However in \cite{DDVV} the authors
make progress in the case of submanifolds in $(n+2)$-dimensional real space
form. In the same paper they conjectured that the above pointwise inequality
is valid for higher dimensional cases (see also \cite{DFV}, \cite{Lu1} and 
\cite{Lu2} for some works on this pointwise inequality).

Recently, G. Jianguan and T. Zizhou give a proof of DDVV conjecture on a
submanifold of a real space form. Furthermore they solved the problem of its
equality case.

Let $M$ a submanifold of a $(n+d)$-dimesional Euclidean space $\mathbb{E}%
^{n+d}.$ Denote by $\overline{R}$ the curvature tensor of the Vander
Waerden-Bortoletti connection $\overline{\nabla }$ of $M$ and $h$ is the
second fundamental form of $M$ in $\mathbb{E}^{n+d}.$ The submanifold $M$ is
called semi-parallel (or semi-symmetric \cite{Lu}) if $\overline{R}\cdot h=0$
\cite{De}. This notion is an extrinsic analogue for semi-symmetric spaces,
i.e. Riemannian manifolds for which $R\cdot R=0$ and a direct generalization
of parallel submanifolds, i.e. submanifolds for which $\overline{\nabla }%
h=0. $ In \cite{De} J. Deprez showed the fact that the submanifold $M\subset 
\mathbb{E}^{n+d}$ is semi-parallel implies that $(M,g)$ is semi-symmetric.
For references on semi-symmetric spaces, see \cite{Sz}; for references on
parallel immersions, see \cite{Fe}. In \cite{De} J. Deprez gave a local
classification of semi-parallel hypersurfaces in Euclidean space. It is
easily seen that all surfaces are semi-parallel.

In the present study we consider the Wintgen ideal surfaces in $n$%
-dimensional Euclidean space $\mathbb{E}^{n}.$ We have shown that Wintgen
ideal surfaces in $\mathbb{E}^{n}$ satisfying the semiparallelity condition $%
\overline{R}(X,Y)\cdot h=0$ are totally umbilical. Further, we obtain some
results in $\mathbb{E}^{4}.$

\section{Basic Concepts}

Let $M$ be a smooth surface in n-dimensional Euclidean space $\mathbb{E}^{n}$
given with the surface patch $X(u,v)$ : $(u,v)\in D\subset \mathbb{E}^{2}$.
The tangent space to $M$ at an arbitrary point $p=X(u,v)$ of $M$ span $%
\left
\{ X_{u},X_{v}\right \} $. In the chart $(u,v)$ the coefficients of
the first fundamental form of $M$ are given by 
\begin{equation}
E=\left \langle X_{u},X_{u}\right \rangle ,F=\left \langle X_{u},X_{v}\right
\rangle ,G=\left \langle X_{v},X_{v}\right \rangle ,  \label{B1}
\end{equation}%
where $\left \langle ,\right \rangle $ is the Euclidean inner product. We
assume that $W^{2}=EG-F^{2}\neq 0,$ i.e. the surface patch $X(u,v)$ is
regular.\ For each $p\in M$, consider the decomposition $T_{p}\mathbb{E}%
^{n}=T_{p}M\oplus T_{p}^{\perp }M$ where $T_{p}^{\perp }M$ is the orthogonal
component of the tangent plane $T_{p}M$ in $\mathbb{E}^{n}$, that is the
normal space of $M$ at $p$.

Let $\chi (M)$ and $\chi ^{\perp }(M)$ be the space of the smooth vector
fields tangent and normal to $M$ respectively. Denote by $\nabla $ and $%
\overline{\nabla }$ the Levi-Civita connections on $M$ and $\mathbb{E}^{n},$
respectively. Given any vector fields $X_{i}$ and $X_{j}$ tangent to $M$
consider the second fundamental map $h:\chi (M)\times \chi (M)\rightarrow
\chi ^{\perp }(M);$%
\begin{equation}
h(X_{i},X_{j})=\widetilde{\nabla }_{X_{i}}X_{j}-\nabla _{X_{i}}X_{j};\text{ }%
1\leq i,j\leq 2.  \label{B2}
\end{equation}%
where $\widetilde{\nabla }$ is the induced. This map is well-defined,
symmetric and bilinear.

For any normal vector field $N_{\alpha }$ $1\leq \alpha \leq n-2$ of $M$,
recall the shape operator $A:\chi ^{\perp }(M)\times \chi (M)\rightarrow
\chi (M);$%
\begin{equation}
A_{N_{\alpha }}X_{i}=-\widetilde{\nabla }_{N_{\alpha
}}X_{i}+D_{X_{i}}N_{\alpha };\text{ \  \ }1\leq i\leq 2.  \label{B3}
\end{equation}%
where $D$ denotes the normal connection of $M$ in $\mathbb{E}^{n}$ \cite{Ch1}%
$.$This operator is bilinear, self-adjoint and satisfies the following
equation:%
\begin{equation}
\left \langle A_{N_{\alpha }}X_{i},X_{j}\right \rangle =\left \langle
h(X_{i},X_{j}),N_{\alpha }\right \rangle \text{, }1\leq i,j\leq 2.
\label{B4}
\end{equation}

The equation (\ref{B2}) is called Gaussian formula, and%
\begin{equation}
h(X_{i},X_{j})=\overset{n-2}{\underset{\alpha =1}{\sum }}h_{ij}^{\alpha
}N_{\alpha },\  \  \  \  \ 1\leq i,j\leq 2  \label{B5}
\end{equation}%
where $h_{ij}^{\alpha }$ are the coefficients of the second fundamental form 
$h$ \cite{Ch1}. If $h=0$ then $M$ is called totally geodesic. $M$ is totally
umbilical if all shape operators are proportional to the identity map. $M$
is an isotropic surface if for each $p$ in $M$, $\left \Vert
h(X,X)\right
\Vert $ is independent of the choice of a unit vector $X$ in $%
T_{p}M$.

If we define a covariant differentiation $\overline{\nabla }h$ of the second
fundamental form $h$ on the direct sum of the tangent bundle and normal
bundle $TM\oplus T^{\bot }M$ of $M$ by 
\begin{equation}
(\overline{\nabla }_{X_{i}}h)(X_{j},X_{k})=D_{X_{i}}h(X_{j},X_{k})-h(\nabla
_{X_{i}}X_{j},X_{k})-h(X_{j},\nabla _{X_{i}}X_{k})  \label{B6}
\end{equation}%
for any vector fields $X_{i}$,$X_{j},X_{k}$ tangent to $M$. Then we have the
Codazzi equation 
\begin{equation}
(\overline{\nabla }_{X_{i}}h)(X_{j},X_{k})=(\overline{\nabla }%
_{X_{j}}h)(X_{i},X_{k})  \label{B7}
\end{equation}%
where $\overline{\nabla }$ is called the Vander Waerden-Bortoletti
connection of $M$ \cite{Ch1}.

We denote $R$ and $\overline{R}$ the curvature tensors associated with $%
\nabla $ and $D$ respectively;%
\begin{eqnarray}
R(X_{i},X_{j})X_{k} &=&\nabla _{X_{i}}\nabla _{X_{j}}X_{k}-\nabla
_{X_{j}}\nabla _{X_{i}}X_{k}-\nabla _{\lbrack X_{i},X_{j}]}X_{k},  \label{B8}
\\
R^{\bot }(X_{i},X_{j})N_{\alpha } &=&h(X_{i},A_{N_{\alpha
}}X_{j})-h(X_{j},A_{N_{\alpha }}X_{i}).  \label{B9}
\end{eqnarray}

The equation of Gauss and Ricci are given respectively by%
\begin{eqnarray}
\text{ \  \ }\left \langle R(X_{i},X_{j})X_{k},X_{l}\right \rangle &=&\left
\langle h(X_{i},X_{l}),h(X_{j},X_{k})\right \rangle -\left \langle
h(X_{i},X_{k}),h(X_{j},X_{l})\right \rangle ,  \label{B10} \\
\text{ \  \  \  \  \ }\left \langle R^{\bot }(X_{i},X_{j})N_{\alpha },N_{\beta
}\right \rangle &=&\left \langle [A_{N_{\alpha }},A_{N_{\beta
}}]X_{i,}X_{j}\right \rangle  \label{B11}
\end{eqnarray}%
for the vector fields $X_{i},X_{j},X_{k}$ tangent to $M$ and $N_{\alpha
},N_{\beta }$ normal to $M$ \cite{Ch1}.

Let us $X_{i}\wedge X_{j}$ denote the endomorphism $X_{k}\longrightarrow
\left \langle X_{j},X_{k}\right \rangle X_{i}-$ $\left \langle
X_{i},X_{k}\right \rangle X_{j}.$ Then the curvature tensor $R$ of $M$ is
given by the equation 
\begin{equation}
R(X_{i},X_{j})X_{k}=\overset{n-2}{\underset{\alpha =1}{\sum }}\left(
A_{N_{\alpha }}X_{i}\wedge A_{N_{\alpha }}X_{j}\right) X_{k}.  \label{B12}
\end{equation}

It is easy to show that 
\begin{equation}
R(X_{i},X_{j})X_{k}=K\left( X_{i}\wedge X_{j}\right) X_{k}.  \label{B13}
\end{equation}%
where $K$ is the Gaussian curvature of $M$ defined by 
\begin{equation}
K=\left \langle h(X_{1},X_{1}),h(X_{2},X_{2})\right \rangle -\left \Vert
h(X_{1},X_{2})\right \Vert ^{2}  \label{B14}
\end{equation}%
(see \cite{GR}).

The normal curvature $K_{N}$ of $M$ is defined by (see \cite{DDVV}) 
\begin{equation}
K_{N}=\left \{ \overset{n-2}{\underset{1=\alpha <\beta }{\sum }}\left
\langle R^{\bot }(X_{1},X_{2})N_{\alpha },N_{\beta }\right \rangle
^{2}\right \} ^{1/2}.  \label{B15}
\end{equation}

We observe that the normal connection $D$ of $M$ is flat if and only if $%
K_{N}=0,$ and by a result of Cartan, this equivalent to the
diagonalisability of all shape operators $A_{N_{\alpha }}$\ of $M$, which
means that $M$ is a totally umbilical surface in $\mathbb{E}^{n}$.

Further, the mean curvature vector $\overrightarrow{H}$ of $M$ is defined by%
\begin{equation}
\overrightarrow{H}=\frac{1}{2}\sum_{\alpha =1}^{n-2}tr(A_{N_{\alpha
}})N_{\alpha }.  \label{B16}
\end{equation}

\section{Wintgen ideal Surfaces}

For the surface \ in \ Euclidean \ 3-space $\  \mathbb{E}^{3}$, \  \ the \
Euler \ inequality $K\leq \left \Vert \overrightarrow{H}\right \Vert ^{2}$
holds. Obviously, $K=\left \Vert \overrightarrow{H}\right \Vert ^{2}$
everywhere on $M$ if and \ only \ if the surface $M$ is totally umbilical in 
$\mathbb{E}^{3}.$ So by \ theorem \ of Meusnier, $M$ is totally umbilical if
\ and \ only if $M$ is part of a plane or a round sphere $S^{2}$ in $\mathbb{%
E}^{3}$ \cite{Ch2}$.$ In 1979, Wintgen proved a basic relation \ between \  \
the \  \ Gauss curvature $K$, the normal curvature $K_{N}$ and a squared mean
curvature $\left \Vert \overrightarrow{H}\right \Vert ^{2}$ of any surface $%
M $ in Euclidean 4-space $\mathbb{E}^{4},$ namely 
\begin{equation}
K+K_{N}\leq \left \Vert \overrightarrow{H}\right \Vert ^{2}  \label{C1}
\end{equation}%
with the equality holding if and only if the curvature ellipse of $M$ is a
circle \cite{GR}.

A surface $M$ in $\mathbb{E}^{4}$ is called \textit{Wintgen ideal} if it
satisfies the equality case of Wintgen inequality identically (\ref{C1})
(see, \cite{DPV}, \cite{PV}).

In \cite{Ch2} B.Y. Chen gave the following result.

\begin{theorem}
\cite{Ch2} Let $M$ be a smooth surface in Euclidean 4-space $\  \mathbb{E}%
^{4}.$ Then \ the \ Wintgen inequality (\ref{C1}) holds at every point in M.
Moreover,

$i)$ If $K_{N}\geqslant 0$ holds at a point $p\in M$, then the equality (\ref%
{C1}) holds at $p$ if and only if, with respect to some \ suitable \
orthonormal \ frame $\left \{ X_{1},X_{2},N_{1},N_{2}\right \} $ at $p$, the
shape operator at $p$ satisfies%
\begin{equation}
A_{N_{1}}=\left( 
\begin{array}{cc}
\lambda _{1}+\mu & 0 \\ 
0 & \lambda _{1}-\mu%
\end{array}%
\right) ,\text{ }A_{N_{2}}=\left( 
\begin{array}{cc}
\lambda _{2} & \mu \\ 
\mu & \lambda _{2}%
\end{array}%
\right) ,  \label{C2}
\end{equation}%
$ii)$ If $K_{N}<0$ holds at a point $p\in M$, then the \ equality \ (\ref{C1}%
) holds at $p$ if and only if, the shape operator at $p$ satisfies%
\begin{equation}
A_{N_{1}}=\left( 
\begin{array}{cc}
\lambda _{1}-2\mu & 0 \\ 
0 & \lambda _{1}%
\end{array}%
\right) ,\text{ }A_{N_{2}}=\left( 
\begin{array}{cc}
\lambda _{2} & \mu \\ 
\mu & \lambda _{2}%
\end{array}%
\right) .  \label{C3}
\end{equation}
\end{theorem}

Wintgen's inequality is extended to surfaces in real space form \ by I.V.
Guadalupe and L. Rodriguez [13]. However in \cite{DDVV} \ the \ authors make
progress in the case of submanifolds \ in $(n+2)$-dimensional real space
form.

In the same paper they conjectured that \ (\textit{DDVV conjecture}) \ the
pointwise inequality (\ref{C1}) is valid for higher dimensional cases (see
also \cite{DFV}, \cite{Lu1} and \cite{Lu2}).

Recently, G. \ Jianguan \ and \ T. Zizhou \ give \ a \ proof \ of \ DDVV
conjecture on a submanifold of a real space \  \  \ form. \ Furthermore they
solved the problem of its equality case;

\begin{theorem}
\cite{JZ} Let $M\subset \mathbb{E}^{n}$ a smooth surface \ given \ with \
the patch $\ X(u,v)$. Then the \ Wintgen inequality (\ref{C1}) holds at
every point in M. Moreover, the equalities holds at some point $p\in M$ if
and only if there exits an \ orthonormal \ basis $\left \{ X,Y\right \} $ of 
$T_{p}M$ and orthonormal basis $\left \{ N_{1},N_{2},...,N_{n-2}\right \} $
of \ $T_{p}^{\bot }M$, such that{} ($r>3$);%
\begin{equation}
A_{N_{1}}=\left( 
\begin{array}{cc}
h_{11}^{1} & 0 \\ 
0 & h_{22}^{1}%
\end{array}%
\right) ,A_{N_{2}}=\left( 
\begin{array}{cc}
\lambda _{2} & \mu \\ 
\mu & \lambda _{2}%
\end{array}%
\right) ,A_{N_{r}}=\left( 
\begin{array}{cc}
\lambda _{r} & 0 \\ 
0 & \lambda _{r}%
\end{array}%
\right) ,\text{ }r\geq 3  \label{C4}
\end{equation}%
where $h_{11}^{1}=\lambda _{1}+\mu ,$ $h_{22}^{1}=\lambda _{1}-\mu $ or $%
h_{11}^{1}=\lambda _{1}-2\mu ,$ $h_{22}^{1}=\lambda _{1}.$ For the first
case (resp. second case) $M$ is called a Wintgen ideal surface of first kind
(resp. second kind).
\end{theorem}

\section{Semi-parallel Surfaces}

Let $M$ a smooth surface in $n$-dimesional Euclidean space $\mathbb{E}^{n}.$
Let $\overline{\nabla }$ be the connection of Vander Waerden-Bortoletti of $%
M $. Denote the tensors $\overline{\nabla }$ by $\overline{R}$ . Then the
product tensor $\overline{R}\cdot h$ of the curvature tensor $\overline{R}$
with the second fundamental form $h$ is defined by%
\begin{eqnarray*}
(\overline{R}(X_{i},X_{j})\cdot h)(X_{k},X_{l}) &=&\overline{\nabla }%
_{X_{i}}(\overline{\nabla }_{X_{j}}h(X_{k},X_{l}))-\overline{\nabla }%
_{X_{j}}(\overline{\nabla }_{X_{i}}h(X_{k},X_{l})) \\
&&-\overline{\nabla }_{[X_{i},X_{j}]}h(X_{k},X_{l})
\end{eqnarray*}%
for all $X_{i},X_{j},X_{k},X_{l}$ tangent to $M.$

The surface $M$ is said to be semi-parallel if $\overline{R}\cdot h=0,$ i.e. 
$\overline{R}(X_{i},X_{j})\cdot h=0$ (\cite{Lu}, \cite{De}, \cite{Ds}, \cite%
{OAM}). It is easy to see that 
\begin{eqnarray}
(\overline{R}(X_{i},X_{j})\cdot h)(X_{k},X_{l}) &=&R^{\bot
}(X_{i},X_{j})h(X_{k},X_{l})  \label{D1} \\
&&-h(R(X_{i},X_{j})X_{k},X_{l})-h(X_{k},R(X_{i},X_{j})X_{l}),  \notag
\end{eqnarray}

This notion is an extrinsic analogue for semi-symmetric spaces, i.e.
Riemannian manifolds for which $R\cdot R=0$ and a generalization of parallel
surfaces, i.e. $\overline{\nabla }h=0$ \cite{Fe}$.$

First, we proved the following result.

\begin{lemma}
Let $M\subset \mathbb{E}^{n}$ a smooth surface given with the patch $X(u,v)$%
. Then the following equalities are hold;%
\begin{eqnarray}
(\overline{R}(X_{1},X_{2})\cdot h)(X_{1},X_{1}) &=&\left( \sum_{\alpha
=1}^{n-2}h_{11}^{\alpha }(h_{22}^{\alpha }-h_{11}^{\alpha })+2K\right)
h(X_{1},X_{2})  \notag \\
&&+\sum_{\alpha =1}^{n-2}h_{11}^{\alpha }h_{12}^{\alpha
}(h(X_{1},X_{1})-h(X_{2},X_{2}))  \notag \\
(\overline{R}(X_{1},X_{2})\cdot h)(X_{1},X_{2}) &=&\left( \sum_{\alpha
=1}^{n-2}h_{12}^{\alpha }(h_{22}^{\alpha }-h_{11}^{\alpha })\right)
h(X_{1},X_{2})  \label{D2} \\
&&+(\sum_{\alpha =1}^{n-2}h_{12}^{\alpha }h_{12}^{\alpha
}-K)(h(X_{1},X_{1})-h(X_{2},X_{2}))  \notag \\
(\overline{R}(X_{1},X_{2})\cdot h)(X_{2},X_{2}) &=&\left( \sum_{\alpha
=1}^{n-2}h_{22}^{\alpha }(h_{22}^{\alpha }-h_{11}^{\alpha })-2K\right)
h(X_{1},X_{2})  \notag \\
&&+\sum_{\alpha =1}^{n-2}h_{22}^{\alpha }h_{12}^{\alpha
}(h(X_{1},X_{1})-h(X_{2},X_{2})).  \notag
\end{eqnarray}
\end{lemma}

\begin{proof}
Substituting (\ref{B5}) and (\ref{B4}) into (\ref{B9}) we get 
\begin{eqnarray}
R^{\bot }(X_{1},X_{2})N_{\alpha } &=&h_{12}^{\alpha
}(h(X_{1},X_{1})-h(X_{2},X_{2}))  \label{D3} \\
&&+(h_{22}^{\alpha }-h_{11}^{\alpha })h(X_{1},X_{2}).  \notag
\end{eqnarray}%
Further, by the use of (\ref{B13}) we get 
\begin{eqnarray}
R(X_{1},X_{2})X_{1} &=&-KX_{2}  \label{D4} \\
R(X_{1},X_{2})X_{2} &=&KX_{1}.  \notag
\end{eqnarray}%
So, substituting (\ref{D3}) and (\ref{D4}) into (\ref{D1}) we get the result.
\end{proof}

Semi-parallel surfaces in $\mathbb{E}^{n}$ are classified by J. Deprez \cite%
{De}:

\begin{theorem}
\cite{De} Let $M$ a surface in $n$-dimensional Euclidean space $\mathbb{E}%
^{n}.$ Then $M$ is semi-parallel if and only if locally;

i) $M$ is equivalent to a 2-sphere, or

ii) $M$ has trivial normal connection, or

iii) $M$ is an isotropic surface in $\mathbb{E}^{5}\subset \mathbb{E}^{n}$
satisfying $\left \Vert H\right \Vert ^{2}=3K.$
\end{theorem}

We get the following result.

\begin{theorem}
Let $M$ a Wintgen ideal surface in $\mathbb{E}^{n}.$ If $M$ is semi-parallel
then it is a totally umbilical surface in $\mathbb{E}^{n}.$
\end{theorem}

\begin{proof}
Let $M$ be a Wintgen ideal surface in $\mathbb{E}^{n}$ given with the patch $%
X(u,v)$. Then by Theorem 3.2 we get 
\begin{eqnarray}
h(X_{1},X_{2}) &=&\mu N_{2},  \label{D5} \\
h(X_{1},X_{1})-h(X_{2},X_{2}) &=&(h_{11}^{1}-h_{22}^{1})N_{1}  \notag
\end{eqnarray}%
where $h_{11}^{1}=\lambda _{1}+\mu ,$ $h_{22}^{1}=\lambda _{1}-\mu $ or $%
h_{11}^{1}=\lambda _{1}-2\mu ,$ $h_{22}^{1}=\lambda _{1}.$ Further,
substituting (\ref{D5}) into (\ref{D2}) and using Lemma 4.1 one can get%
\begin{eqnarray}
(\overline{R}(X_{1},X_{2})\cdot h)(X_{1},X_{1}) &=&\left( \sum_{\alpha
=1}^{n-2}\sum_{\alpha =1}^{d}h_{11}^{\alpha }(h_{22}^{\alpha
}-h_{11}^{\alpha })+2K\right) \mu N_{2}  \notag \\
&&+\left( \sum_{\alpha =1}^{n-2}h_{11}^{\alpha }h_{12}^{\alpha }\right)
(h_{11}^{1}-h_{22}^{1})N_{1}  \notag \\
(\overline{R}(X_{1},X_{2})\cdot h)(X_{1},X_{2}) &=&\left( \sum_{\alpha
=1}^{n-2}h_{12}^{\alpha }(h_{22}^{\alpha }-h_{11}^{\alpha })\right) \mu N_{2}
\label{D6} \\
&&+\left( \sum_{\alpha =1}^{n-2}h_{12}^{\alpha }h_{12}^{\alpha }-K\right)
(h_{11}^{1}-h_{22}^{1})N_{1}  \notag \\
(\overline{R}(X_{1},X_{2})\cdot h)(X_{2},X_{2}) &=&\left( \sum_{\alpha
=1}^{n-2}h_{22}^{\alpha }(h_{22}^{\alpha }-h_{11}^{\alpha })-2K\right) \mu
N_{2}  \notag \\
&&+\sum_{\alpha =1}^{n-2}h_{22}^{\alpha }h_{12}^{\alpha
}(h_{11}^{1}-h_{22}^{1})N_{1}  \notag
\end{eqnarray}%
Since $h_{12}^{2}=\mu ,h_{12}^{\alpha }=0,\alpha \neq 2,$ then the equation (%
\ref{D6}) becomes%
\begin{eqnarray}
(\overline{R}(X_{1},X_{2})\cdot h)(X_{1},X_{1}) &=&h_{11}^{2}\mu
(h_{11}^{1}-h_{22}^{1})N_{1}  \notag \\
&&+[h_{11}^{1}(h_{22}^{1}-h_{11}^{1})+2K]\mu N_{2},  \notag \\
(\overline{R}(X_{1},X_{2})\cdot h)(X_{1},X_{2}) &=&(\mu
^{2}-K)(h_{11}^{1}-h_{22}^{1})N_{1}  \label{D7} \\
&&+(h_{22}^{1}-h_{11}^{1})\mu ^{2}N_{2},  \notag \\
(\overline{R}(X_{1},X_{2})\cdot h)(X_{2},X_{2}) &=&h_{22}^{2}\mu
(h_{11}^{1}-h_{22}^{1})N_{1}  \notag \\
&&+[h_{22}^{1}(h_{22}^{1}-h_{11}^{1})-2K]\mu N_{2}.  \notag
\end{eqnarray}%
Suppose that, $M$ is semi-parallel then by definition $(\overline{R}%
(X_{1},X_{2})\cdot h)(X_{i},X_{j})=0,(1\leq i,j\leq 2).$ So, we get%
\begin{eqnarray}
\mu ^{2}(h_{11}^{1}-h_{22}^{1}) &=&0,  \notag \\
\lambda _{2}\mu (h_{11}^{1}-h_{22}^{1}) &=&0,  \notag \\
(\mu ^{2}-K)(h_{11}^{1}-h_{22}^{1}) &=&0,  \label{D8} \\
\mu \lbrack h_{11}^{1}(h_{22}^{1}-h_{11}^{1})+2K] &=&0,  \notag \\
\mu \lbrack h_{22}^{1}(h_{22}^{1}-h_{11}^{1})-2K] &=&0.  \notag
\end{eqnarray}%
where $h_{11}^{1}=\lambda _{1}+\mu ,$ $h_{22}^{1}=\lambda _{1}-\mu $ or $%
h_{11}^{1}=\lambda _{1}-2\mu ,$ $h_{22}^{1}=\lambda _{1}.$ Now, we study the
case $h_{11}^{1}=h_{22}^{1}.$ Then $\mu =0$ by (\ref{D8}). This means that $%
R^{\bot }=0$ by (\ref{D3}) and (\ref{D5}). This is equivalent to say that $M$
has vanishing normal curvature $K_{N}$, which means that $M$ is a totally
umbilical surface in $\mathbb{E}^{m}$.
\end{proof}

\section{Some Results in $\mathbb{E}^{4}$}

Rotation surfaces were studied in \cite{Vr} by Vranceanu as surfaces in $%
\mathbb{E}^{4}$ which are defined by the following parametrization;%
\begin{eqnarray}
X(u,v) &=&(r(v)\cos v\cos u,r(v)\cos v\sin u,  \label{E1} \\
&&r(v)\sin v\cos u,r(v)\sin v\sin u)  \notag
\end{eqnarray}%
where $r(v)$ is a real valued non-zero function.

We choose a moving frame $\left \{ X_{1},X_{2},N_{1},N_{2}\right \} $ such
that $X_{1},X_{2}$ are tangent to $M$ and $N_{1},N_{2}$ are normal to $M$ as
given the following (see \cite{Yo}):

\begin{eqnarray*}
X_{1} &=&\frac{\partial }{r(v)\partial u}=(\text{-}\cos v\sin u,\cos v\cos u,%
\text{-}\sin v\sin u,\sin v\cos u), \\
X_{2} &=&\frac{\partial }{A\partial v}=\frac{1}{A}(B(v)\cos u,B(v)\sin
u,C(v)\cos u,C(v)\sin u), \\
N_{1} &=&\frac{1}{A}(-C(v)\cos u,-C(v)\sin u,B(v)\cos u,B(v)\sin u), \\
N_{2} &=&(-\sin v\sin u,\sin v\cos u,\cos v\sin u,-\cos v\cos u)
\end{eqnarray*}%
where 
\begin{eqnarray*}
A(v) &=&\sqrt{r^{2}(v)+(r^{\prime })^{2}(v)}, \\
B(v) &=&r^{\prime }(v)\cos v-r(v)\sin v, \\
C(v) &=&r^{\prime }(v)\sin v+r(v)\cos v.
\end{eqnarray*}

Furthermore, by covariant differentiation with respect to $X_{1}$ and $X_{2}$
a straightforward calculation gives:%
\begin{eqnarray}
\widetilde{\nabla }_{X_{1}}X_{1} &=&-a(v)k(v)X_{2}+a(v)N_{1},  \notag \\
\widetilde{\nabla }_{X_{2}}X_{2} &=&b(v)N_{1},  \label{E2} \\
\widetilde{\nabla }_{X_{2}}X_{1} &=&-a(v)N_{2},  \notag
\end{eqnarray}%
where%
\begin{eqnarray}
k(v) &=&\frac{r^{\prime }(v)}{r(v)},  \notag \\
a(v) &=&\frac{1}{\sqrt{r^{2}(v)+(r^{\prime })^{2}(v)}},  \label{E3} \\
b(v) &=&\frac{2(r^{\prime }(v))^{2}-r(v)r^{\prime \prime }(v)+r^{2}(v)}{%
(r^{2}(v)+(r^{\prime })^{2}(v))^{3/2}}  \notag
\end{eqnarray}%
are differentiable functions.

Thus by the use of (\ref{B9}) together with (\ref{B14}) and (\ref{B15}) we
get the following result.

\begin{proposition}
Let $M$ a Vranceanu surface given with the surface patch (\ref{E1}). Then
the Gaussian curvature $K$ coincides with the normal curvature $K_{N}$ of $%
M. $ That is ;%
\begin{equation}
K=K_{N}=a(v)b(v)-a^{2}(v).  \label{E4}
\end{equation}
\end{proposition}

\begin{theorem}
Let $M$ a Vranceanu surface given with the surface patch (\ref{E1}). Then, $%
M $ is Wintgen ideal surface of first kind if and only if 
\begin{equation}
r(v)=\pm \sqrt{c_{1}\cos (2v)-c_{2}\sin (2v)},\text{ }c_{1},c_{2}\in R
\label{E5}
\end{equation}%
holds.
\end{theorem}

\begin{proof}
Suppose that the Vranceanu surface $M$ is given with the surface patch (\ref%
{E1}). If $M$ is a Wintgen ideal surface of first kind then by (\ref{C4})%
\begin{equation}
a(v)=\lambda _{1}+\mu ,b(v)=\lambda _{1}-\mu ,\mu =-a(v)  \label{E6}
\end{equation}%
holds. Further, from (\ref{E3}) and (\ref{E6}) we get%
\begin{equation*}
(r^{\prime }(v))^{2}+r(v)r^{\prime \prime }(v)+2r^{2}(v)=0
\end{equation*}%
which has a nontrivial solution (\ref{E5}). Conversely, if the equality (\ref%
{E5}) holds then the Vranceanu surface becomes Wintgen ideal surface of
first kind.
\end{proof}

\begin{theorem}
Let $M$ be a Vranceanu surface given with the surface patch (\ref{E1}).
Then, $M$ is Wintgen ideal surface of second kind if and only if $M$ is a
minimal surface satisfying 
\begin{equation}
r(v)=\pm \frac{1}{\sqrt{c_{1}\sin (2v)-c_{2}\cos (2v)}},  \label{E7}
\end{equation}%
where $c_{1}$ and $c_{2}$ are real constants.
\end{theorem}

\begin{proof}
Suppose that the Vranceanu surface $M$ is given with the surface patch (\ref%
{E1}) is Wintgen ideal surface of first kind then by the use of (\ref{C4})
the following equalities hold;%
\begin{equation}
a(v)=\lambda _{1}-2\mu ,\text{ }b(v)=\lambda _{1},\text{ }\mu =-a(v)
\label{E8}
\end{equation}%
Further, from(\ref{E3}) and (\ref{E8}) we get 
\begin{equation}
3(r^{\prime }(v))^{2}-r(v)r^{\prime \prime }(v)+2r^{2}(v)=0  \label{E9}
\end{equation}%
which has a nontrivial solution (\ref{E7}). Conversely, if the equality(\ref%
{E7}) holds then the Vranceanu surface becomes Wintgen ideal surface of
second kind.
\end{proof}

\begin{corollary}
Let $M$ a Vranceanu surface given with the surface patch (\ref{E1}). If $M$
is semi-parallel then $M$ is a flat surface satisfying $%
r(v)=c_{1}e^{c_{2}v}. $
\end{corollary}

\begin{proof}
Suppose the Vranceanu surface $M$ is semi-parallel then by the use of (\ref%
{D2}) with (\ref{E2}) we get%
\begin{eqnarray*}
(\overline{R}(X_{1},X_{2})\cdot h)(X_{1},X_{1}) &=&\left( 3a^{2}(v)\left(
a(v)-b(v)\right) \right) N_{2} \\
(\overline{R}(X_{1},X_{2})\cdot h)(X_{1},X_{2}) &=&\left( a(v)\left(
a(v)-b(v)\right) (2a(v)-b(v))\right) N_{1} \\
(\overline{R}(X_{1},X_{2})\cdot h)(X_{2},X_{2}) &=&a(v)\left(
3a(v)b(v)-2a(v)^{2}-b(v)^{2}\right) N_{2}.
\end{eqnarray*}%
Suppose that, $M$ is semi-parallel then by (\ref{D1}) $(\overline{R}%
(X_{1},X_{2})\cdot h)(X_{i},X_{j})=0,$ $(1\leq i,j\leq 2).$ Which implies
that $a(v)-b(v)=0.$ So, by (\ref{E4}) $K=K_{N}=0.$ Further, from (\ref{E3})
we get the result.
\end{proof}

\end{document}